\documentclass[11pt]{article}
\usepackage{caption}
\usepackage{subcaption}
\usepackage{cite}
\usepackage{graphics}
\usepackage{authblk}
\usepackage{amsmath,amssymb,amsfonts}
\usepackage[english]{babel}
\usepackage{hyperref}
\usepackage{array}
\usepackage[rightcaption]{sidecap}
\usepackage{longtable}
\usepackage{enumitem}
\usepackage{scalerel}
\usepackage{amsthm}
\usepackage{graphicx}
\usepackage{color}
\usepackage{algorithm}
\usepackage{algorithmic}
\usepackage{comment}
\usepackage{orcidlink} 
\numberwithin{equation}{section}

\usepackage{array}        
\usepackage{booktabs}     
\usepackage{graphicx}    
\usepackage{caption}    

\definecolor{darkgreen}{rgb}{0,0.7,0}

\usepackage{dsfont}

\newcommand{\RR}{\mathbb{R}}

\newcommand{\norm}[1]{\Vert #1 \Vert}

\newcolumntype{M}[1]{>{\centering\arraybackslash}m{#1}}

\usepackage[english]{babel}
\providecommand{\keywords}[1]{\textbf{{Keywords:}} #1}
\numberwithin{equation}{section}
\usepackage{color}

\usepackage{tikz}
\usetikzlibrary{positioning}
\usetikzlibrary{arrows.meta,shadows,positioning}
\usetikzlibrary{shapes.multipart}
\usetikzlibrary{backgrounds}
\tikzset{
  frame/.style={
    rectangle, draw,
    text width=6em, text centered,
    minimum height=4em,drop shadow,fill=white,
    rounded corners,
  },
  line/.style={
    draw, -{Latex},rounded corners=3mm,
  }
}

\newtheorem{remark}{Remark}[section]

\newtheorem{theorem}{Theorem}[section]

\newtheorem{proposition}{Proposition}[section]

\title{\textbf{Some linear feedback laws for stabilisation of sterile insect technique control system }}
\author[]{Kala Agbo bidi \large\orcidlink{0009-0005-0695-2381} \thanks{Sorbonne Universit\'{e}, CNRS, Universit\'{e} Paris-Cit\'{e}, Laboratoire Jacques-Louis Lions, LJLL,  INRIA  \'{e}quipe CAGE, F-75005 Paris, France. \texttt{kala.agbo\_bidi@sorbonne-universite.fr}}} 
\date{\empty}

\begin{document}

\maketitle

\begin{abstract}
\textbf{The implementation of the Sterile Insect Technique (SIT) to manage a target population has been the focus of numerous recent scientific studies. The present work focuses on a feedback law that depends linearly on the state variables of the SIT control system. We provide both mathematical proof and numerical illustrations demonstrating the global asymptotic stability of the population to zero when releasing a number of sterile insects proportional to different state variables of the SIT model.
 }

\end{abstract}
\noindent \keywords{Sterile Insect Technique, Pest control,  Dynamical control system, Feedback design, Lyapunov global stabilization, Mosquito population control, Vector-borne disease.}

\section{Introduction}
In recent decades, the use of the Sterile Insect 
Technique (SIT) has been extended to the management 
of mosquito species such as \textit{Aedes}, \textit{Culex}, 
and \textit{Anopheles}, which are vectors of various diseases, including \textit{malaria, Zika, dengue}, and other \textit{arbovirosis}. The mathematical modeling of this method leads to a controlled dynamical system, whose study aims at designing release strategies to achieve specific objectives, such as the asymptotic stabilization of the target population close to extinction, the optimization of release costs, and the robustness of control against environmental variations or system parameter disturbances.
 The SIT mathematical control system that we consider in this work is introduced in \cite{strugarek2019use}.  In some previous works \cite{2024agbobidi,bidi_global_2025,2023Rossi}, the backstepping method was applied to construct a pivotal feedback law that robustly ensures the stabilisation of the control system to zero. Recently, another pivotal feedback law has been proposed in \cite{2024-Bliman-hal}. The application of these feedback laws requires state measurement, which can be challenging. We address this problem in our previous work by constructing an observer model that requires only the states of sterile males and wild males to estimate all other state variables (see \cite{2024agbobidi}) and by designing a feedback law that depends on the adult population using Reinforcement Learning of the SIT control system \cite{bidi2023reinforcement}.
Linear feedback laws are less flexible than non-linear laws, but they are robust and easier to implement in practice. It would be very useful to establish whether the quantity of sterile males released to globally stabilise the population to extinction needs to be \textit{proportional} to the density of wild males, total females or eggs, or some other variable state.

Some mathematical progress in this direction was made in our previous work \cite{bidi_global_2025}, where we aimed at proving that releasing a quantity of sterile males proportional to the wild male population, with a coefficient of proportionality satisfying certain conditions, successfully achieves global stabilisation of the target population to zero. However, proving global stability remains a challenging task.

The structure of this paper is as follows. The mathematical model studied in this paper is presented in Section~\ref{sec:Model}, while our main contributions are discussed in Section~\ref{sec:LFS}.
In Section~\ref{sec:feedbackEMMs} we present the stability results of the linear feedback law depending on $ E $ and $ M + M_s $. 
In Section~\ref{sec:feedLEM} we study a linear law including the states $ E $ and $ M $. Numerical simulations are provided to illustrate all these results.

\section{Sterile insecte technique control system}\label{sec:Model}
The mathematical modelling of the mosquito life cycle after neglecting the Allee effect is presented in \cite{strugarek2019use}. Its state variables are $E$ which is the density of the aquatic stage (which we will often just designate as eggs but which also includes larvae and pupal), $M$ the density of the wild adult males and $F$ the density of adult females mated by wild males. The dynamical system is 
\begin{gather}\label{eq:completeODElifecycle}
     \begin{aligned}
    &\dot{E} = \beta_E F \left(1-\frac{E}{K}\right) - \big( \nu_E + \delta_E \big) E,  \\
    &\dot{F} = \nu\nu_E E  - \delta_F F, \\
    &\dot{M} = (1-\nu)\nu_E E - \delta_M M,
\end{aligned}
\end{gather}
where $\beta_E>0$ is the oviposition rate, $\delta_E,\delta_M,\delta_F >0$ are the death rates for eggs, wild adult males and fertilised females respectively, $\nu_E>0$ is the hatching rate for eggs, $\nu\in (0,1)$ is the probability that a pupa will give rise to a female, and $(1-\nu)$ is therefore the probability that it will give rise to a male.  And, to simplify, we assume that females are fertilised as soon as they emerge from the pupal stage. $K>0$ is the environmental capacity for the aquatic phase. Table \ref{tab:tableparamsRL} shows the  values of these parameters that we consider in our work. The basic offspring number for this system is given by
\begin{equation}
    R = \frac{\beta_E\nu\nu_E}{\delta_F(\delta_E+\nu_E)}.
\end{equation}
 This dynamics admits a non-trivial equilibrium called the persistence equilibrium, which we denote by $X_{E^*}$, where
\begin{gather}\label{persistence equilibrium}
    \begin{aligned}
    &E^*:=  K(1-\frac{1}{R}),\\&
    X_{E^*} := (
        E^*, \; \frac{\nu\nu_E}{\delta_F}E^*,\;\frac{(1-\nu)\nu_E}{\delta_M} E^*)^T.
    \end{aligned}
\end{gather}

\newcolumntype{C}[1]{>{\centering\arraybackslash}m{#1}}
\begin{table}
    \centering
		\setlength{\tabcolsep}{0.01cm}
		\begin{tabular}{C{1cm} C{6cm} C{3cm} C{2.3cm} C{2cm}}
			\toprule
			  &  \textbf{Parameter name} & \textbf{Typical interval} & \textbf{Value in our work} & \textbf{Unit} \\
			\midrule
			$\beta_E$ & Effective fecundity &[7.46, 14.85]& 8& Day$^{-1}$\\
			$\nu_E$ & Hatching parameter &[0.005, 0.25]& 0.05&Day$^{-1}$\\
			$\delta_E$& Aquatic
			phase death rate & [0.023, 0.046]&0.03&Day$^{-1}$\\
			$\delta_F$& Female death rate &[0.033, 0.046]& 0.04&Day$^{-1}$\\
			$\delta_M$ & Male death rate &[0.077, 0.139]&  0.1&Day$^{-1}$\\
			$\delta_s$ & Sterilized male death rate& -& 0.12&Day$^{-1}$\\
			$\nu$ & Probability of emergence& -& 0.49&\\
   K & Environmental capacity for eggs &- &50000&\\

			\bottomrule
		\end{tabular}
  \caption{Parameters for the system \eqref{eq:completeODESIT}. This includes typical value ranges, which can be found in a population of Aedes polynesiensis in French Polynesia, as well as our chosen values in this work. Further details can be found in \cite{strugarek2019use,hughes2013modelling,riviere1988ecologie,suzuki1978breeding,chambers2011male,hapairai2014effect,hapairai2013population,hapairai2013studies}. 
  }
    \label{tab:tableparamsRL}
\end{table}
When sterile males are released into the population according to a control function denoted by $u$, the compartment of females mated by sterile males, denoted by $F_s$, is added to the system along with $M_s$, their evolutionary density. This leads to the following control system introduced in \cite{strugarek2019use} (once the Allee effect is neglected):

\begin{subequations}\label{eq:completeODESIT}
\begin{align}
    &\dot{E} = \beta_E F \left(1-\frac{E}{K}\right) - \big( \nu_E + \delta_E \big) E,  \label{eqEsit}\\
    &\dot{F} = \nu\nu_E E \frac{M}{M+\gamma_s M_s} - \delta_F F, \label{eqFsit}\\
     &\dot{M} = (1-\nu)\nu_E E - \delta_M M,\label{eqMsit} \\
    &\dot{F}_s = \nu\nu_E E \frac{\gamma M_s}{M+\gamma M_s} - \delta_F F_s, \label{eqFssit}\\
    &\dot{M}_s = u(t) - \delta_s M_s. \label{eqMssit}
\end{align}
\end{subequations}
Here $0<\gamma_s\leq1$ accounts for the fact that females may have a preference for fertile males and  the probability that a female mates with a fertile male is ${M}/{(M + \gamma M_s)}$.

\begin{proposition}
    Let $$ t \mapsto  (E(t),F(t),M(t),F_s(t),M_s(t))^T $$ be a solution of \eqref{eq:completeODESIT}, defined at time $ t = 0 $ and satisfying $$ x_0=(E(0),F(0),M(0),F_s(0),M_s(0))^T \in \mathbb{R}_+^5. $$ Then it is defined on $ [0,+\infty) $. Moreover, if $ E(0) \geq K $, then there exists a unique time $ t_0 \geq 0 $ such that $ E(t_0) = K $, and one has:
\begin{gather}
E(t) < K \quad \forall t > t_0.
\end{gather}
\end{proposition}
Let us consider
\begin{gather}
B_K:=\left\{(E,F,M,F_s,M_s)^T\in [0,+\infty)^5:\; E\leq K\right\}.
\end{gather}
Note that for all $u\geq 0$, $B_K$ is positively invariant for \eqref{eq:completeODESIT}. We define 
\begin{equation}
    t \mapsto x(x_0,t) := (E(t),F(t),M(t),F_s(t),M_s(t))^T,
\end{equation}
the solution of \eqref{eq:completeODESIT} with initial condition $x_0$.
In this paper, all numerical simulations are performed under the initial condition 
\begin{gather}\label{intialcond}
    x_0= (E(0),F(0),M(0),F_s(0),M_s(0))^T =(X_{E^*},0,0)^T.
\end{gather} 
\section{Linear feedback stabilization}\label{sec:LFS}
To enhance the robustness of the control strategy, we assume  that the sterilization and release process may reduce the longevity of the released sterile males. Consequently we assume
\begin{gather}\label{deltas>deltaM}
    \delta_s> \delta_M.
\end{gather}
\subsection{Linear  feedback laws depending on eggs density and total male density}\label{sec:feedbackEMMs}

Let us define $ \hat\delta := \delta_s - \delta_M $. 
For any $ \psi > 0 $, we consider a control function:
\begin{equation}\label{EMMs}
    u(x) = \psi (1 - \nu)\nu_E E + \hat\delta (M + M_s).
\end{equation}

After computing the non-trivial equilibrium point, we obtain the following condition on $ \psi $:  
\begin{gather}
    \psi \geq \frac{R - 1}{\gamma} - \frac{\hat\delta}{\delta_M},
\end{gather}
ensuring that the extinction equilibrium $ \mathbf{0} $ is the only equilibrium of   system \eqref{eq:completeODESIT} with the feedback law \eqref{EMMs}. In the following proposition we present  a result concerning the closed-loop system, which is important in proving the asymptotic stability result above.

\begin{proposition}\label{see:t>psi}
    Let $\psi>0 $. Let   $$t\in [0,+\infty) \mapsto x(x_0,t)= (E(t),F(t),M(t),F_s(t),M_s(t))^T$$  be a  solution of system \eqref{eq:completeODESIT} with the feedback law \eqref{EMMs} and   initial condition $ x_0\in B_K$. Then, for  all $t\geq \frac{\psi}{\hat\delta} $, $$M_s(t)\geq \psi M(t).$$
\end{proposition}
\begin{proof}
Let us consider $t\mapsto (E(t),F(t),M(t),F_s(t),M_s(t))^T$ solution of the  system \eqref{eq:completeODESIT} with the feedback law \eqref{EMMs} for  initial data in $B_K$.
    Substituting \eqref{EMMs} in \eqref{eqMssit}, gives
\begin{equation*}
    \dot M_s = \psi (1-\nu)\nu_E E+\hat\delta M -\delta_MM_s.
\end{equation*}
For all $t\geq 0$ recall that
\begin{equation}
    M(t) = M^0e^{-\delta_Mt} + (1-\nu)\nu_E e^{-\delta_Mt}\int_0^t E(s)e^{\delta_M s}\, ds.
\end{equation}
    \begin{gather}\label{Ms-psicE}
        \begin{aligned}
        M_s(t) &= M_s^0e^{-\delta_st} + \hat\delta M^0 te^{-\delta_Mt} \\ &+ \hat\delta (1-\nu)\nu_E e^{-\delta_M t} \int_0^t (t-s)E(s) e^{\delta_M s} \, ds\\ &+ \psi (1-\nu)\nu_E e^{-\delta_M t} \int_0^t E(s) e^{\delta_M s} \, ds.
    \end{aligned}
    \end{gather}
    
We deduce that $M_s(t)-\psi M(t)\geq   M^0 e^{-\delta_Mt}(\hat\delta t-\psi) $.
Hence for all $t\geq \frac{\psi}{\hat\delta}$, $M_s(t)\geq \psi M(t)$ and this concludes the proof.
\end{proof}
\begin{remark}
    This result shows that when the sterile male is released according to the feedback law \eqref{EMMs} in the system \eqref{eq:completeODESIT}, for all $t\geq \frac{\psi}{\hat\delta}$, the density of the sterile male becomes greater than the male density for some chosen proportionality coefficient $\psi$. Referring to \cite{2024-Bliman-hal,bidi_global_2025}, if $\psi$ is greater than some value depending on the number of offspring $R$, the population converges to the extinction equilibrium. 
\end{remark}

Let $\mathcal{H} :\RR_+\longrightarrow \RR_+$  be the  function
$ \mathcal{H} :p\mapsto \frac{R}{1+\gamma p}$. Note that 
$\psi>\frac{R-1}{\gamma}$ implies $\mathcal{H}(\psi)<1$.
\begin{theorem}\label{thm:GASEMMs}
    We assume that 
    \begin{gather}\label{psi>R}
        \psi> \frac{R-1}{\gamma}.
    \end{gather}
    Then, $\textbf{0}$ is globally exponentially stable for the  system \eqref{eq:completeODESIT} with the feedback law \eqref{EMMs} in $[0,+\infty)^5$. 
More precisely, for every $c_r\in [0,c_e)$ with
\begin{gather}\label{c-e}
\small
    \begin{aligned}
    c_e :=  \min\left\{
    \frac{\nu\nu_E(1-\mathcal{H}(\psi))}{1+\mathcal{H}(\psi)}, \delta_M,
    \frac{\beta_E\delta_F(1-\mathcal{H}(\psi))}{2},\delta_F,\delta_s, (\nu_E+\delta_E),\delta_F\right\},
    \end{aligned}
\end{gather}
there exists $C>0$ such that, for every solution
$$t\in [0,+\infty) \mapsto x(x_0,t)=(E(t),F(t),M(t),F_s(t),M_s(t))^T$$  of \eqref{eq:completeODESIT} with the feedback law \eqref{EMMs} and  $ x_0\in [0,+\infty)^5$, one has
\begin{gather}\label{expconvofx(t)}
    \norm{x(x_0,t)}\leq C\norm{x_0}e^{-c_rt}\;\forall\; t\geq 0.
\end{gather}
\end{theorem}

\begin{proof}
We consider $$t\mapsto x(x_0,t) := (E(t),F(t),M(t),F_s(t),M_s(t))^T$$ the solution of \eqref{eq:completeODESIT} for some $x_0\in B_K$.
Let us consider the subsystem for \eqref{eqEsit}, \eqref{eqFsit}, and \eqref{eqMsit}. Note that this subsystem does not depend on \eqref{eqFssit} and \eqref{eqMssit}. We denote by $t\mapsto z(t):= (E(t),F(t),M(t))^T$ its  solution. We will prove that when $\psi$ satisfies condition \eqref{psi>R}, any solution $t\mapsto z(t)$ converges exponentially to $(0,0,0)^T$ with some convergence rate.
To prove this result, we consider the function $V$ defined by (see \cite{bidi_global_2025}):
\begin{equation}
    \begin{aligned}
        &V:  z\in [0,+\infty)^3  \mapsto V(z)\in \mathbb{R}_+, \\
        &V(z) := \frac{1+\mathcal{H}(\psi)}{1-\mathcal{H}(\psi)} E + M + \frac{2\beta_E}{\delta_F(1-\mathcal{H}(\psi))}F.
    \end{aligned}
\end{equation}
We observe that this function satisfies the following condition, which is important for proving the exponential stability:
\begin{equation}\label{cmVcM}
    Q_1\norm{z}\leq V(z)\leq Q_2\norm{z} \;\forall\; z\geq (0,+\infty)^3.
\end{equation}
where using \eqref{psi>R}
\begin{gather}
    Q_1 := \min \left\{\frac{1+\mathcal{H}(\psi)}{1-\mathcal{H}(\psi)}  ,1 , \frac{2\beta_E}{\delta_F(1-\mathcal{H}(\psi))}\right\}>0,\\
    Q_2 := \max \left\{\frac{1+\mathcal{H}(\psi)}{1-\mathcal{H}(\psi)}  ,1 , \frac{2\beta_E}{\delta_F(1-\mathcal{H}(\psi))}\right\}>0.
\end{gather}
Thus,
\begin{gather}
    \text{V is of class } \mathcal{C}^1,\\
    V(z)>V((0,0,0)^T)=0, \;\forall z \in [0, +\infty)^3\setminus\{(0,0,0)^T\},\\
     \text{$V(z)\to +\infty$ as $\norm{z}\to +\infty$.}
\end{gather}

Thus, $V$ is a candidate to be a Lyapunov function. It remains to prove that there exists a constant $c>0$ such that $\dot{V}(z(t))\leq -c V(z(t))$ for all $t\geq 0$. We start by computing its derivative:
\begin{align*}
    \dot{V}(z) &= \nabla V(z)\cdot g(z)\\&= \begin{pmatrix}\frac{1+\mathcal{H}(\psi)}{1-\mathcal{H}(\psi)}\\  1 \\ \frac{2\beta_E}{\delta_F(1-\mathcal{H}(\psi))}
\end{pmatrix}
\cdot
\begin{pmatrix} \beta_E F \left(1-\frac{E}{K}\right) - (\delta_E+\nu_E) E\\  (1-\nu)\nu_E E - \delta_M M\\   \frac{\nu\nu_E M}{M+\gamma M_s} E - \delta_F F
\end{pmatrix}.
\end{align*}
\begin{align*}
    \dot{V}(z) &= \beta_E \frac{1+\mathcal{H}(\psi)}{1-\mathcal{H}(\psi)} F-\delta_M M\\& -\frac{1+\mathcal{H}(\psi)}{1-\mathcal{H}(\psi)}\frac{\beta_E}{K} EF-\frac{2\beta_E\delta_F}{\delta_F(1-\mathcal{H}(\psi))} F \\&- (\nu_E+\delta_E)\frac{1+\mathcal{H}(\psi)}{1-\mathcal{H}(\psi)}E + (1-\nu)\nu_E E \\&+  \frac{2\beta_E\nu\nu_E}{\delta_F(1-\mathcal{H}(\psi))}\frac{ M}{M+\gamma M_s}E .
\end{align*}
From Proposition \ref{psi>R}, we deduce that for all $t \geq T_0 := \frac{\psi}{\hat\delta}$, we have 
\begin{equation}
    \frac{M}{M+\gamma M_s} \leq \frac{1}{1+\gamma \psi}.
\end{equation}
Thus, for all $t \geq T_0$,
\begin{equation*}
    \dot{V}(z) \leq  -\beta_E F - \delta_M M - \frac{1+\mathcal{H}(\psi)}{1-\mathcal{H}(\psi)}\frac{\beta_E}{K} EF - (\nu\nu_E + \delta_E)E.
\end{equation*}

Using \eqref{psi>R}, we deduce the existence of $c_a > 0$ such that for all $t \geq T_0$
\begin{align}
    \dot{V}(z) \leq -c_a V(z), \quad \forall z \in [0, +\infty)^3,
\end{align}
where
\begin{gather*}
    c_a = \min\left\{
    \frac{\nu\nu_E(1-\mathcal{H}(\psi))}{1+\mathcal{H}(\psi)}, \delta_M,
    \frac{\beta_E\delta_F(1-\mathcal{H}(\psi))}{2}\right\}.
\end{gather*}

We deduce from \eqref{cmVcM} that there exists $C > 0$ such that
\begin{gather}
    \norm{z(t)} \leq C\norm{z(T_0)} e^{-c_a(t-T_0)}, \quad \forall t \geq T_0.
\end{gather}
Note that for all $t \in [0,T_0]$, from the system \eqref{eq:completeODESIT} with the feedback law \eqref{EMMs}, we deduce that 
\begin{align}\label{ddtz(t)}
    \frac{d}{dt}\norm{z(t)} \leq C\norm{z(t)}.
\end{align}
In  \eqref{ddtz(t)} and until the end of this proof, $C$ denotes  constants which may vary form place to place but are independent of $x$.
Hence
\begin{gather}
    \norm{z(T_0)} \leq C\norm{z(0)}.
\end{gather}
Thus, there exists $C > 0$ such that 
\begin{gather*}\label{z<=z0}
    \norm{z(t)} \leq C\norm{z(0)} e^{-c_at}, \quad \forall t \geq 0.
\end{gather*}
From \eqref{eqMssit} and \eqref{EMMs}, we deduce that if $c_a = \delta_M$, then
\begin{align*}
    M_s(t) &\leq M_s^0 e^{-\delta_s t} + \hat\delta M^0 t e^{-\delta_M t} \\ &\quad + \hat\delta C\norm{z(0)} (1-\nu)\nu_E t^2 e^{-\delta_M t} \\ &\quad + \psi C\norm{z(0)} (1-\nu)\nu_E t e^{-\delta_M t}.
\end{align*}
Otherwise,
\begin{equation}\label{solms(t)}
    \begin{aligned}
    M_s(t) &\leq M_s^0 e^{-\delta_s t} + \hat\delta M^0 t e^{-\delta_M t} \\
    &\quad + \frac{\hat\delta C\norm{z(0)} (1-\nu)\nu_E}{(\delta_M - c_a)^2}\\& (e^{-c_a t} - e^{-\delta_M t} - (\delta_M - c_a)t e^{-\delta_M t}) \\
    &\quad + \frac{\psi C\norm{z(0)} (1-\nu)\nu_E}{\delta_M - c_a} (e^{-c_a t} - e^{-\delta_M t}).
    \end{aligned}
\end{equation}
We deduce that for $c_m<c_a$ there exists $C > 0$ independent of $x$ such that for all $t \geq 0$
\begin{gather}
    M_s(t) \leq C\norm{x(0)}e^{-c_m t}.
\end{gather}
From \eqref{eqFssit}, we deduce that
\begin{align}\label{solfs(t)}
    F_s(t) \leq F_s^0 e^{-\delta_F t} + \frac{C \nu \nu_E \norm{z(0)}}{\delta_F - c_a} (e^{-c_a t} - e^{-\delta_F t}).
\end{align}
We deduce that for $c_m<\min\{c_a,\delta_F\}$, there exists $C > 0$ independent of $x$ such that for all $t \geq 0$
\begin{gather}
    F_s(t) \leq C\norm{x(0)}e^{-c_m t}.
\end{gather}
Hence, for all $x_0 \in B_K$, there exists a constant $C > 0$ such that all solution $t \mapsto x(x_0,t)$ of \eqref{eq:completeODESIT} with the feedback law \eqref{EMMs} satisfies
\begin{gather}
    \norm{x(x_0,t)} \leq C\norm{x_0}e^{-c_m t}, \quad \forall t \geq 0.
\end{gather}
If $x_0\notin B_K$  then there exists a unique time $ t_0 \geq 0 $ such that $ E(t_0) = K $, and one has:
\begin{gather}
E(t) < K \quad \forall t > t_0.
\end{gather}
For all $t\in [0,t_0]$ one has
\begin{equation}\label{EFM<}
    \begin{aligned}
    &E(t)\leq E^0 e^{-(\delta_E+\nu_E) t},\\&
    F(t) \leq F^0e^{-\delta_Ft} + \frac{\nu\nu_EE^0}{\delta_F-(\delta_E+\nu_E)}(e^{-(\nu_E+\delta_E)t}-e^{-\delta_Ft}),
    \\&
    M(t) \leq M^0e^{-\delta_Mt} + \frac{(1-\nu)\nu_EE^0}{\delta_M-\delta_F}(e^{-\delta_Ft}-e^{-\delta_Mt}),\\&
    F_s(t) \leq F_s^0e^{-\delta_Ft} + \frac{\nu\nu_EE^0}{\delta_F-(\delta_E+\nu_E)}(e^{-(\nu_E+\delta_E)t}-e^{-\delta_Ft}),
\end{aligned}
\end{equation}
Substitution \eqref{EFM<}  and \eqref{EMMs} in \eqref{eqMssit}, we prove that  for every $$c_w <  \min\{\delta_M,\delta_F,\delta_s, (\nu_E+\delta_E)\}$$ there exists $C>0$ such that 
\begin{equation}
    \norm{M_s(t)}\leq C \norm{x_0} e^{-c_w t},\;\forall t\in [0,t_0],
\end{equation}
and therefore, using also \eqref{EFM<}
\begin{equation}
    \norm{x(x_0,t)}\leq C \norm{x_0} e^{-c_w t},\;\forall t\in [0,t_0].
\end{equation}
 Hence we conclude that when $\psi$ satisfies \eqref{psi>R}, for all $c_r<c_e$ with $c_e$ defined in \eqref{c-e},
 there exists $C>0$ such that for all  $ x_0\in [0,+\infty)^5$ 
\begin{gather}
    \norm{x(x_0,t)}\leq C\norm{x_0}e^{-c_rt}\;\forall\; t\geq 0.
\end{gather}
\end{proof}
The biological interpretation of this result is as follows: assuming that females exhibit no preference for fertile males (i.e. $\gamma=1$), Theorem \ref{thm:GASEMMs} guarantees that the following linear control law
\begin{equation}
    u(x) = 2R(1-\nu)\nu_E E + (\delta_s-\delta_M) (M + M_s)
\end{equation}
globally asymptotically stabilizes the sterile control system to $\textbf{0}$. 
  We run the dynamics as long as $E > 1$ and the result is shown in  Figure \ref{fig:EMMs}.
\begin{figure}[H]
		\centering
		\includegraphics[width=\textwidth]{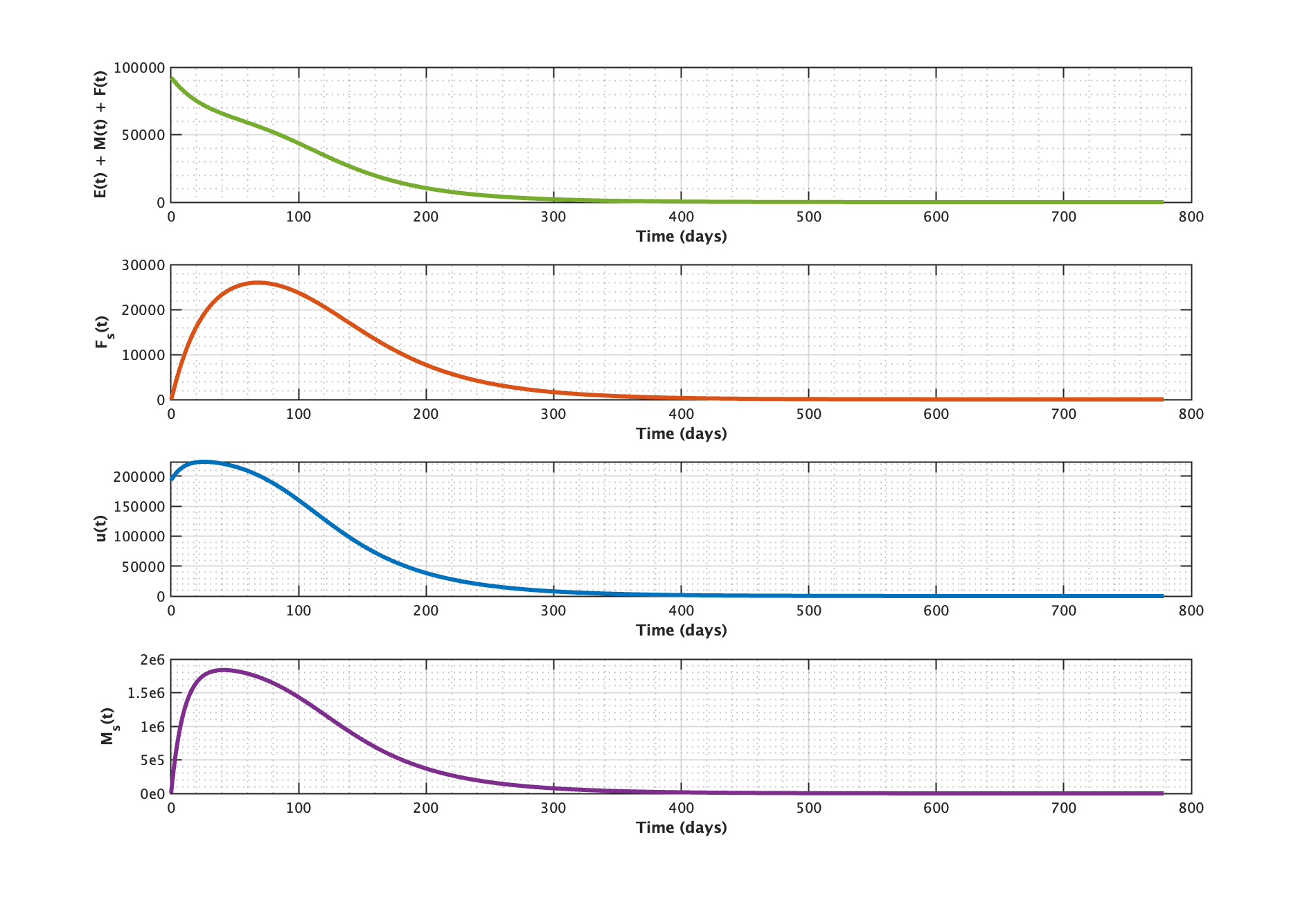}
    \caption{Plot of the  system  \eqref{eq:completeODESIT} with the feedback law \eqref{EMMs} for $\psi = 2R$. According to the parameters fixed in Table \ref{tab:tableparamsRL}, we have $R = 76.56$}
    \label{fig:EMMs}
\end{figure}

\subsection{Linear  feedback laws depending on aquatic phase and wild male densities}
\label{sec:feedLEM}

We consider the  control function

\begin{gather}\label{deltaEMu=}
    u(x) = \alpha M + (1-\nu)\nu_E\sigma E.
\end{gather}
\begin{proposition}\label{propEM}
Let $$t\mapsto x(x_0,t)=(E(t),F(t),M(t),F_s(t),M_s(t))^T$$ be a solution to the  system \eqref{eq:completeODESIT} with the feedback law \eqref{deltaEMu=}  and initial data $x_0\in B_K$.
    For any $\sigma>0$,  for all $\alpha>\sigma\hat\delta$, one has $\forall\;t\geq T_e=\frac{-\ln(1-\frac{(\delta_s-\delta_M)\sigma}{\alpha})}{\alpha}$
    \begin{gather}
         M_s(t)\geq \sigma M(t).
    \end{gather}
\end{proposition}
\begin{proof}

Let 
 $$t\mapsto x(x_0,t)=(E(t),F(t),M(t),F_s(t),M_s(t))^T$$  be a solution of  system \eqref{eq:completeODESIT} with the feedback law \eqref{deltaEMu=}  and initial data $x_0\in B_K$. 
 for all $t\geq 0$
     \begin{align*}
        M_s(t) &= M_s^0e^{-\delta_st} + \frac{\alpha M^0}{\hat\delta }e^{-\delta_Mt}(1-e^{-\hat\delta t})+ (1-\nu)\nu_E \sigma e^{-\delta_s t} \int_0^t E(s) e^{\delta_s s} \, ds\\&+ \frac{(1-\nu)\nu_E \alpha}{\hat\delta} e^{-\delta_M t} \int_0^t E(s) e^{\delta_M s} \, ds - \frac{(1-\nu)\nu_E \alpha }{\hat\delta}  e^{-\delta_s t} \int_0^t E(s) e^{\delta_s s} \, ds, 
    \end{align*}
and
    \begin{equation*}
        \sigma M(t) = \sigma M^0e^{-\delta_Mt}+ \sigma (1-\nu)\nu_Ee^{-\delta_Mt}\int_0^t E(s)e^{\delta_Ms} ds.
    \end{equation*}
    Hence,
\begin{align*}
    M_s(t)-\sigma M(t)&\geq M^0 e^{-\delta_Mt} \bigg(\frac{\alpha}{\hat\delta}(1-e^{-\hat\delta t})-\sigma\bigg) \\&+ (1-\nu)\nu_E  \int_0^t E(s) \bigg( e^{-(t-s)\delta_s}- e^{-(t-s)\delta_M}\bigg)\bigg(\sigma-\frac{\alpha}{\hat\delta}\bigg) \, ds.
\end{align*}
From \eqref{deltas>deltaM},
  $\hat\delta>0$ and
    \begin{align*}
         \int_0^t E(s) \bigg( e^{-(t-s)\delta_s}- e^{-(t-s)\delta_M}\bigg)\bigg(\sigma-\frac{\alpha}{\hat\delta}\bigg) \, ds \geq 0,
    \end{align*}
     is equivalent to
    $\alpha \geq  \sigma\hat\delta$.
    Note that if $\alpha= \sigma\hat\delta$,   $\frac{\alpha}{\hat\delta}(1-e^{-\hat\delta t})-\sigma<0$.  For $\alpha>\sigma\hat\delta$,
    $\frac{\alpha}{\hat\delta}(1-e^{-\hat\delta t})-\sigma\geq 0$ is equivalent to $t>\frac{-\ln(1-\frac{(\delta_s-\delta_M)\sigma}{\alpha})}{\alpha}$. We conclude that by choosing $\alpha>\sigma\hat\delta$, for any $\sigma>0$, one has $M_s(t)\geq \sigma M(t)$ for all $t\geq T_e=\frac{-\ln(1-\frac{(\delta_s-\delta_M)\sigma}{\alpha})}{\alpha}$.\\

  This concludes the proof. 
\end{proof}
\begin{theorem}\label{thm:GASME}
   We assume 
    \begin{gather}\label{sigma>R}
        \sigma> \frac{R-1}{\gamma}.
    \end{gather}
    Then, $\textbf{0}$ is globally exponentially stable  for the  system \eqref{eq:completeODESIT} with the feedback law \eqref{deltaEMu=} in $[0,+\infty)^5$. More precisely, for every $c_p\in [0,c_b)$ with
\begin{gather}\label{c-b}
\small
    \begin{aligned}
    c_b :=  \min\left\{
    \frac{\nu\nu_E(1-\mathcal{H}(\sigma))}{1+\mathcal{H}(\sigma)}, \delta_M,
    \frac{\beta_E\delta_F(1-\mathcal{H}(\sigma))}{2},\delta_F,\delta_s, (\nu_E+\delta_E),\delta_F\right\},
    \end{aligned}
\end{gather}
there exists $C>0$ such that, for every solution
$$t\in [0,+\infty) \mapsto x(x_0,t)=(E(t),F(t),M(t),F_s(t),M_s(t))^T$$  of \eqref{eq:completeODESIT} with the feedback law \eqref{EMMs} and  $ x_0\in [0,+\infty)^5$, one has
\begin{gather}
    \norm{x(x_0,t)}\leq C\norm{x_0}e^{-c_bt}\;\forall\; t\geq 0.
\end{gather}
\end{theorem}
\begin{proof}
The proof of  this theorem is similar to the proof of Theorem \ref{thm:GASEMMs} using Proposition \ref{propEM} instead of Proposition \ref{see:t>psi}.
\end{proof}
The biological interpretation of this result is as follows: assuming that females exhibit no preference for fertile males (i.e. $\gamma=1$), Theorem \ref{thm:GASME} guarantees that the linear control law
\begin{equation}
    u(x) = 4R(\delta_s-\delta_M)M + 2R(1-\nu)\nu_E E,
\end{equation}
  globally asymptotically stabilizes the target population to $\textbf{0}$  when the death rate of sterile male released is higher than that of wild males. We run the system as long as $E > 1$.

\begin{figure}[H]
		\centering
		\includegraphics[width=\textwidth]{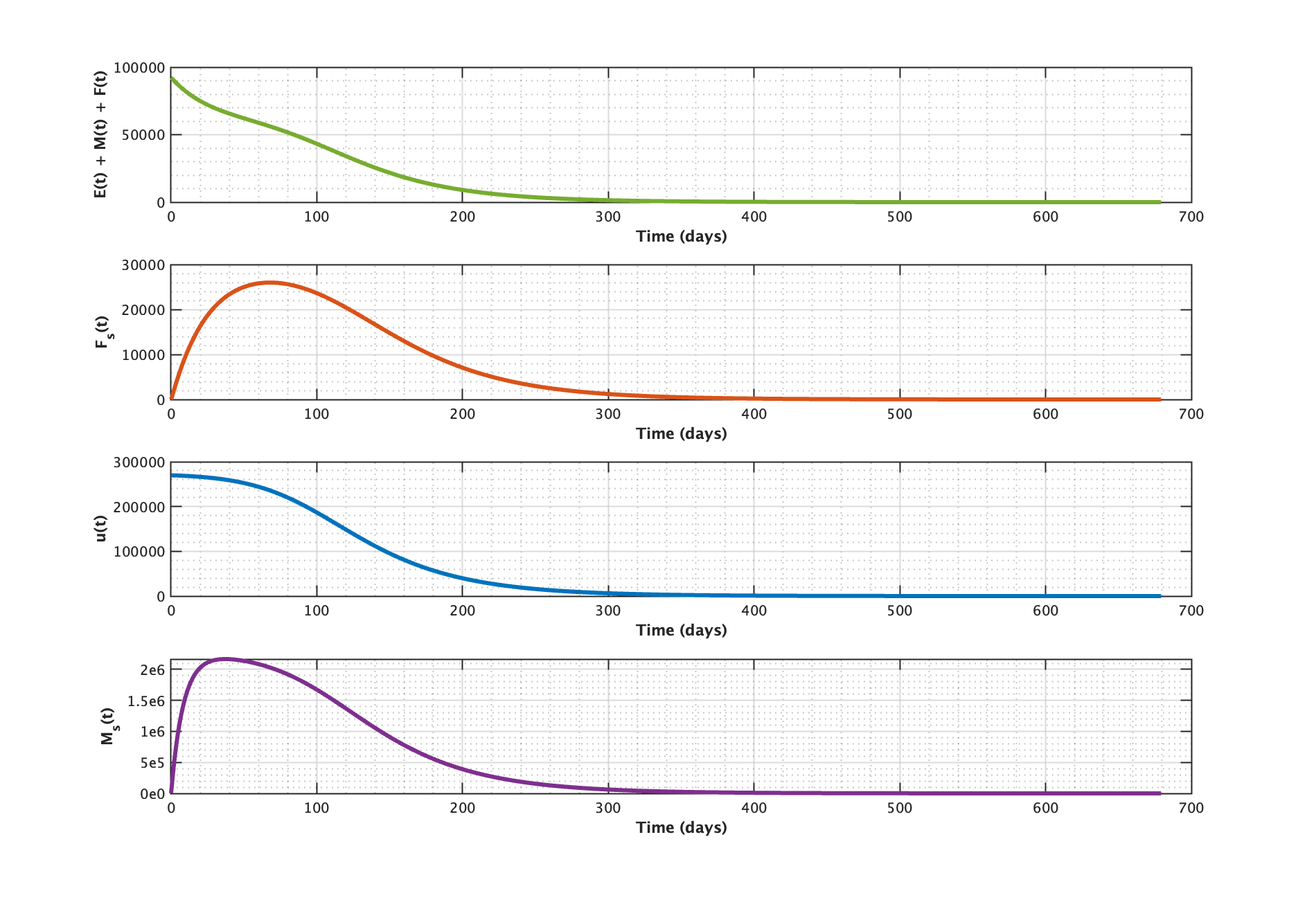}
    \caption{Plot of the  system  \eqref{eq:completeODESIT} with the feedback law \eqref{deltaEMu=} for $\sigma = 2R$, $R=76.56$.}
    \label{fig:uEMglobal}
\end{figure}

\section{Conclusion}

The global asymptotic stability of the SIT control system has been proved for many nonlinear feedback control laws (see \cite{2024agbobidi,bidi_global_2025,2023Rossi,2024-Bliman-hal}). In this work, we prove this result for some simple linear laws. 

We study linear feedback laws depending on the number of wild male mosquitoes $M$ and egg density $E$ in Section \ref{sec:feedLEM}. Thanks to Theorem \ref{thm:GASME}, for instance, the feedback law  $u = 4R(1-\nu)\nu_E E + 2R(\delta_s-\delta_M)M$
stabilizes the system when females have no preference for fertile males ($\gamma=1$). In practice, we use ovitraps to measure the egg density in the population. To estimate the wild male density, we use a technique called release-recapture: sterile males are marked before being released into the population. To measure the adult population, specific traps are placed in different areas to capture them. Thanks to their markings, wild males can be distinguished from sterile ones. 

We also study linear feedback laws depending on the total number of male mosquitoes, $M+M_s$, and egg density $E$ in Section \ref{sec:feedbackEMMs}. Thanks to Theorem \ref{thm:GASEMMs}, for instance, the feedback law  $u = 2R(1-\nu)\nu_E E + (\delta_s-\delta_M)(M+M_s)$
stabilizes the system to zero under the assumption that females have no preference for fertile males ($\gamma=1$). The advantage of this linear law is that it depends on the total number of males in the dynamics. It does not require distinguishing wild male mosquitoes from sterile ones in the traps. 
  
\section*{Acknowledgements}
The author wishes to thank Luis Almeida and Jean-Michel Coron for having drawn his attention to this problem and for the many  enlightening discussions during this work.

\bibliographystyle{plain}
\bibliography{bibliographie}

\end{document}